\documentclass[reqno]{amsart}
\usepackage{amsfonts,amssymb,amsmath,latexsym,indentfirst,cite}
\usepackage{amsthm, amscd}
\usepackage{color,graphicx}
\usepackage{eucal}

\usepackage{ae,enumerate}
\usepackage{tikz}
\newcommand{\R}{\mathbb R}

\newcommand{\peq}{\hspace*{0.10in}}

\newtheorem{theorem}{Theorem}[section]
\newtheorem{proposition}[theorem]{Proposition}
\newtheorem{remark}[theorem]{Remark}
\newtheorem{lemma}[theorem]{Lemma}

\begin{document}
\vglue-1cm \hskip1cm
\title[Global Well-posedness and blow-up for the INLS equation]{Global well-posedness and blow-up on the energy space for the Inhomogeneous Nonlinear Schr\"odinger Equation}

%\author[]{}
%\address{}
%\email{}
%\thanks{}
%\begin{abstract}
%\end{abstract}

%\maketitle

\author[L. G. Farah]{Luiz G. Farah}
\address{ICEx, Universidade Federal de Minas Gerais, Av. Ant\^onio Carlos, 6627, Caixa Postal 702, 30123-970,
Belo Horizonte-MG, Brazil}
\email{lgfarah@gmail.com}
%\thanks{The author was partially supported by CNPq/Brazil and FAPEMIG/Brazil.}

\begin{abstract}
 We consider the supercritical inhomogeneous nonlinear Schr\"odinger equation (INLS)
 $$i\partial_t u+\Delta u+|x|^{-b}|u|^{2\sigma}u=0,$$
where $(2-b)/N<\sigma<(2-b)/(N-2)$ and $0<b<\min\{2,N\}$. 
We prove a Gagliardo-Nirenberg type estimate and use it to establish sufficient conditions for global existence and blow-up in $H^1(\mathbb{R}^N)$. 
\end{abstract}

\maketitle

\section{Introduction}\label{introduction}

In this work we consider the initial value problem (IVP) associated with the supercritical inhomogeneous nonlinear Schr\"odinger equation (INLS)
\begin{equation}\label{INLS}
\begin{cases}
i\partial_t u+\Delta u+|x|^{-b}|u|^{2\sigma}u=0, \;\;x\in\R^N, \;t>0, \\
u(x,0)=u_0(x).
\end{cases}
\end{equation}

This model arises naturally in nonlinear optics for the propagation of laser
beams. The case $b=0$ is the classical nonlinear Schr\"odinger equation studied by several authors in the past years.

A more general form of equation \eqref{INLS}, namely,
$$
i\partial_t u+\Delta u+k(x)|u|^{2\sigma}u=0,
$$
was considered by Merle \cite{M96} and Rapha\"el and Szeftel \cite{RS11} where they study the problem of existence/nonexistence of minimal mass blow-up solutions. However, in both papers, the authors assume that $k(x)$ is bounded which is not verified in our case.

Another type of the inhomogeneous nonlinear Schr\"odinger equation (INLS), but with nonlinearity of the form $|x|^{b}|u|^{2\sigma}u$ with $b>0$, was studied by Chen and Guo \cite{CG07} and Chen \cite{C10}. In these papers the authors obtain certain conditions for global existence and blow-up in the set of radial symmetric functions in $H^1(\mathbb{R}^N)$. The method of the proof was based in variational arguments, however the precise threshold level is not given in terms of ground state solutions.

Before review the state of the art concerning the IVP \eqref{INLS}, let us recall the best Sobolev index where we can expect well-posedness for this model. First, note that if $u(x,t)$ is a solution of \eqref{INLS} so is $u_{\lambda}=\lambda^{\frac{2-b}{2\sigma}}u(\lambda x, \lambda^2 t)$, for all $\lambda >0$. Computing the homogeneous Sobolev norm we have
$$
\|u_\lambda(\cdot,0)\|_{\dot{H}^s}=\lambda^{s+\frac{2-b}{2\sigma}-\frac{N}{2}}\|u_0\|_{\dot{H}^s}.
$$
Thus the critical Sobolev index is given by $s_{\sigma}=N/2-(2-b)/2\sigma$. In this paper, we are interest in the $L^2$-supercritical and $H^1$-subcritical case. Therefore we restrict our attention to the cases where $0<s_{\sigma}<1$. Rewriting this last condition in terms of $\sigma$ we obtain
$$
\dfrac{2-b}{N}<\sigma<2^{\ast},
$$
where $2^{\ast}=(2-b)/(N-2)$, if $N\geq 3$ or $2^{\ast}=\infty$, if $N=1,2$. To avoid $\sigma$ to be negative, we also assume the technical restriction $0<b<\min\{2,N\}$.

Furthermore, the INLS equation \eqref{INLS} has the following conserved quantities
\begin{equation}\label{MC}
Mass\equiv M[u(t)]=\int_{\R^N} |u(x,t)|^2\, dx
\end{equation}
and
\begin{equation}\label{EC}
Energy\equiv E[u(t)]=\frac{1}{2}\int_{\R^N}|\nabla u(x,t)|^2\;dx-\frac{1}{2\sigma+2}\int_{\R^N} |x|^{-b}|u(x,t)|^{2\sigma+2}\;dx.
\end{equation}

The well-posedness theory for the INLS equation \eqref{INLS} was already studied by Genoud and Stuart \cite{GS08} (see also references therein). Using the abstract theory developed by Cazenave \cite{C03}, the authors proved that the IVP \eqref{INLS} is well-posed in $H^1(\R^N)$
\begin{itemize}
\item locally if $0<\sigma < 2^{\ast}$;
\item globally for any initial data in $H^1(\R^N)$ if $0<\sigma < (2-b)/N$;
\item globally for small initial data if $(2-b)/N \leq \sigma < 2^{\ast}$.
\end{itemize}

Moreover, in the limiting case $\sigma=(2-b)/N$ ($L^2$-critical INLS equation) Genoud \cite{G2012} showed how small should be the initial data to have global well-posedness. Indeed, he proved global well-posedness in $H^1(\R^N)$ assuming
\begin{equation}\label{smallness}
\|u_0\|_{L^2(\R^N)}< \|Q\|_{L^2(\R^N)},
\end{equation}
where $Q$ is the unique non-negative, radially-symmetric, decreasing solution of the equation
\begin{equation}\label{ground2}
\Delta Q-Q+|x|^{-b}|Q|^{\frac{2(2-b)}{N}}Q=0.
\end{equation}

Genoud's result is in fact an extension for the INLS model of the classical global well-posedness result proved by Weinstein \cite{W83} for the NLS equation. In \cite{W83} the author proved that solutions for the $L^2$-critical NLS equation (equation \eqref{INLS} with $b=0$ and $\sigma=2/N$) are global in $H^1(\R^N)$ if we assume the smallness condition \eqref{smallness}, where in this case $Q$ stands to be the solution of equation \eqref{ground2} with $b=0$. 

Another extension of Weinstein's result was obtained by Holmer and Roudenko \cite{HR07} for the $L^2$-supercritical and $H^1$-subcritical NLS equation (see also Holmer and Roudenko \cite{HR08} and  Duyckaerts, Holmer and Roudenko \cite{DHR08}). Indeed, the authors established sufficient conditions on the initial data to obtain global and blow-up solutions in $H^1(\R^N)$. More precisely, they prove the following 
\begin{theorem}[\!\!\cite{HR07}]\label{HR}
Let $u(t)$ be an $H^1(\R^N)$ solution to \eqref{INLS} with $b=0$ and $s_{\sigma}=N/2-1/\sigma$. Suppose
$$
E[u_0]^{s_{\sigma}} M[u_0]^{1-s_{\sigma}} < E[Q]^{s_{\sigma}} M[Q]^{1-s_{\sigma}}.
$$
\begin{itemize}
\item[(i)] If $\|\nabla u_0\|_{L^2(\R^N)}^{s_{\sigma}}\|u_0\|_{L^2(\R^N)}^{1-s_{\sigma}} < \|\nabla Q\|_{L^2(\R^N)}^{s_{\sigma}}\|Q\|_{L^2(\R^N)}^{1-s_{\sigma}}$ then the solution $u$ is globally defined.
\item [(ii)] If $\|\nabla u_0\|_{L^2(\R^N)}^{s_{\sigma}}\|u_0\|_{L^2(\R^N)}^{1-s_{\sigma}} > \|\nabla Q\|_{L^2(\R^N)}^{s_{\sigma}}\|Q\|_{L^2(\R^N)}^{1-s_{\sigma}}$ and $u_0$ has finite variance, i.e. $|x|u_0\in L^2(\R^N)$, then the solution $u$ blows-up in finite time.
\end{itemize}
\end{theorem}

Our main interest here is to prove a similar result for the INLS equation \eqref{INLS}. We start showing that the quantity \eqref{EC} is well-defined for functions in $H^1(\R)$. This is guaranteed by the following sharp Gagliardo-Nirenberg inequality
\begin{theorem}\label{best}
Let $\frac{2-b}{N}<\sigma<\frac{2-b}{N-2}$ and $0<b<\min\{2,N\}$,  then the Gagliardo-Nirenberg inequality
\begin{equation}\label{GN}
\int_{\R^N} |x|^{-b}|u(x)|^{2\sigma+2}\;dx \le K_{\rm opt}\,\|\nabla u\|_{L^2(\R^N)}^{N\sigma +b}\|u\|_{L^2(\R^N)}^{2\sigma+ 2-(N\sigma +b)},
\end{equation}
holds, and the sharp constant $K_{\rm opt}>0$ is
explicitly given by
\begin{equation}\label{opt1}
K_{\rm opt}=\left(\frac{N\sigma+b}{2\sigma+2-(N\sigma+b)}\right)^{\frac{2-(N\sigma+b)}{2}} \frac{2\sigma+2}{(N\sigma+b)\|Q\|_{L^2(\R^N)}^{2\sigma}},
\end{equation}
where $Q$ is the unique non-negative, radially-symmetric, decreasing solution of the equation
\begin{equation}\label{ground}
\Delta Q-Q+|x|^{-b}|Q|^{2\sigma}Q=0.
\end{equation}

Moreover the solution $Q$ satisfies the following relations
\begin{equation}\label{groundR1}
\|\nabla Q\|_{L^2(\R^N)}=\left(\dfrac{N\sigma+b}{2\sigma+2-(N\sigma+b)}\right)^{1/2}\|Q\|_{L^2(\R^N)}
\end{equation}
and
\begin{equation}\label{groundR2}
\int_{\R^N} |x|^{-b}|Q(x)|^{2\sigma+2}\;dx=\left(\dfrac{2\sigma+2}{2\sigma+2-(N\sigma+b)}\right)\|Q\|^2_{L^2(\R^N)}
\end{equation}
\end{theorem}

\begin{remark}
The existence and uniqueness of the ground state solution $Q$ for equation \eqref{ground} was proved by Toland \cite{To84} and Yanagida \cite{Ya91} (see also Genoud and Stuart \cite{GS08}). These results hold under the assumptions $0<b<\min\{2,N\}$ and $0<\sigma<2^{\ast}$.
\end{remark}

\begin{remark}
A similar sharp Gagliardo-Nirenberg estimate was also obtained by Chen and Guo \cite{CG07}, for radial symmetric functions in $H^1(\mathbb{R}^N)$, in the case $b<0$ and space dimension $N\geq 2$.
\end{remark}

Next we state our main global well-posedness result.

\begin{theorem}\label{global3}
Let $\frac{2-b}{N}<\sigma<\frac{2-b}{N-2}$, $0<b<\min\{2,N\}$ and set $s_{\sigma}=\frac{N}{2}-\frac{2-b}{2\sigma}$. Suppose that $u(t)$ is the solution of \eqref{INLS} with initial data $u_0\in H^1(\R^N)$ satisfying  
\begin{equation}\label{GR1}
E[u_0]^{s_{\sigma}} M[u_0]^{1-s_{\sigma}} < E[Q]^{s_{\sigma}} M[Q]^{1-s_{\sigma}} 
\end{equation}
and
\begin{equation}\label{GR2}
\|\nabla u_0\|_{L^2(\R^N)}^{s_{\sigma}}\|u_0\|_{L^2(\R^N)}^{1-s_{\sigma}} < \|\nabla Q\|_{L^2(\R^N)}^{s_{\sigma}}\|Q\|_{L^2(\R^N)}^{1-s_{\sigma}},
\end{equation}
then $u(t)$ is a global solution in $H^1(\R^N)$.

Moreover, for any $t\in \R$  we have
\begin{equation}\label{GR3}
\|\nabla u(t)\|_{L^2(\R^N)}^{s_{\sigma}}\|u(t)\|_{L^2(\R^N)}^{1-s_{\sigma}}
< \|\nabla Q\|_{L^2(\R^N)}^{s_{\sigma}}\|Q\|_{L^2(\R^N)}^{1-s_{\sigma}},
\end{equation}
where $Q$ is unique positive even solution of the elliptic equation \eqref{ground}.
\end{theorem}

This theorem can be viewed as an unified global theory result for both INLS and NLS models. Indeed, if $b=0$ we deduce Holmer and Roudenko's result \cite{HR07}, if $\sigma=(2-b)/N$ this is Genoud's result \cite{G2012} and, finally, if $b=0$ and $\sigma=2/N$ we obtain the classical global well-posedness theorem proved by Weinstein\cite{W83}.

The second part of this work is devoted to find blow-up solutions for the INLS equation \eqref{INLS}. Assuming finite variance of the initial data, i.e. $u_0\in H^1(\R^N)\cap \{u: |x|u \in L^2(\R^N)\}$, we prove the following result. 

\begin{theorem}\label{blowup}
Let $\frac{2-b}{N}<\sigma<\frac{2-b}{N-2}$, $0<b<\min\{2,N\}$ and set $s_{\sigma}=\frac{N}{2}-\frac{2-b}{2\sigma}$. Suppose that $u(t)$ is the solution of \eqref{INLS} with initial data 
$$
u_0\in H^1(\R^N)\cap \{u: |x|u \in L^2(\R^N)\}
$$ 
satisfying  
\begin{equation}\label{BR1}
E[u_0]^{s_{\sigma}} M[u_0]^{1-s_{\sigma}} < E[Q]^{s_{\sigma}} M[Q]^{1-s_{\sigma}}
\end{equation}
and
\begin{equation}\label{BR2}
\|\nabla u_0\|_{L^2(\R^N)}^{s_{\sigma}}\|u_0\|_{L^2(\R^N)}^{1-s_{\sigma}} > \|\nabla Q\|_{L^2(\R^N)}^{s_{\sigma}}\|Q\|_{L^2(\R^N)}^{1-s_{\sigma}},
\end{equation}
then the maximum existence time is finite and blow-up in $H^1(\R^N)$ must occur.
\end{theorem}

\begin{remark}
The above Theorem shows that condition \eqref{GR2} is sharp for global existence except for the threshold level $\|\nabla u_0\|_{L^2(\R^N)}^{s_{\sigma}}\|u_0\|_{L^2(\R^N)}^{1-s_{\sigma}} = \|\nabla Q\|_{L^2(\R^N)}^{s_{\sigma}}\|Q\|_{L^2(\R^N)}^{1-s_{\sigma}}$. For the $L^2$-critical INLS equation ($\sigma=(2-b)/N$) threshold solutions was also studied by Genoud \cite{G2012}.  He proved the existence of critical mass blow-up solutions ($\|u_0\|_{L^2(\R^N)}=\|Q\|_{L^2(\R^N)}$), using a pseudo-conformal transformation. It is an interesting open problem to obtain a similar result for the $L^2$-supercritical and $H^1$-subcritical NLS and INLS equations.

%Threshold solutions was studied by Duyckaerts and Roudenko \cite{DH10} for the 3D cubic NLS equation. In this work, the authors exhibited 3 special solutions and proved that all solutions at the threshold behave exactly as this solutions, up to symmetries. It is an interesting open problem to obtain a similar result for the INLS equation \eqref{INLS}.??? 
\end{remark}

\begin{remark}
We point out that if the initial data has negative energy, then via the sharp Gagliardo-Nirenberg inequality we have \eqref{BR2} (see Proposition \ref{Parabola2}). Also note that we can compute the energy \eqref{EC} of the ground state solution $Q$ in terms of its $L^2$ and $\dot{H}^1$ norms. Indeed, by relations \eqref{groundR1} and \eqref{groundR2} a simple computation shows
\begin{equation}\label{Eground}
E[Q]=\dfrac{N\sigma+b-2}{2(2\sigma+2-(N\sigma+b))}\| Q\|^2_{L^2(\R^N)}= \dfrac{N\sigma+b-2}{2(N\sigma+b)}\|\nabla Q\|^2_{L^2(\R^N)}.
\end{equation}

Since $\sigma>(2-b)/N$ it is clear that $E[Q]>0$. Therefore, for initial data with negative energy, we automatically have \eqref{BR1} and the conclusion of Theorem \ref{blowup} also holds in this case.

The identities \eqref{Eground} will be also useful in the proof of our main global well-posedness result Theorem \ref{global3}.
\end{remark}

After this work was completed, we have learned that recently Zhu \cite[Theorem 4.1]{Zhu14} have reached similar global well-posedness and blow-up results for radial symmetric initial data in $H^1(\mathbb{R}^N)$ in the case $b<0$ and space dimension $N\geq 3$. However, the threshold obtained by Zhu does not depend directly on the solutions of the equation \eqref{ground}. Thus, it is not clear how to deduce Theorem \ref{HR} directly from his results.

The plan of this paper is as follows. In the next section we introduce some notation and show the sharp Gagliardo-Nirenberg inequality \eqref{GN}. Next, in section \ref{global}, we prove our global existence result stated in Theorem \ref{global3}. In Section \ref{Virial}, we prove some Virial type identities for solutions of equation \eqref{INLS}. Finally, Section \ref{blowup2} is devoted to our blow-up result stated in Theorem \ref{blowup}.

%%%%%%%%%%%%%%%%%%%%%%%%%%%%%%%%%%%%%%%%%%%%%%%%%%%%%%%%%%%%%%%%%%%%%%%%%%%%%%%%%%%%%%%%%%%%%%%%%%%%%%%%%%%%%%%%%%%%
%%%%%%%%%%%%%%%%%%%%%%%%%%%%%%%%%%%%%%%%%%%%%%%%%%%%%%%%%%%%%%%%%%%%%%%%%%%%%%%%%%%%%%%%%%%%%%%%%%%%%%%%%%%%%%%%%%%%
%%%%%%%%%%%%%%%%%%%%%%%%%%%%%%%%%%%%%%%%%%%%%%%%%%%%%%%%%%%%%%%%%%%%%%%%%%%%%%%%%%%%%%%%%%%%%%%%%%%%%%%%%%%%%%%%%%%%

\section{The sharp Gagliardo-Nirenberg inequality}\label{sharpGN}

Let us start this section by introducing the notation used throughout the paper. We use $\|\cdot\|_{L^p(\R^N)}$ to denote the $L^p(\R^N)$ norm with $p\geq 1$. If necessary, we use subscript to inform which variable we are concerned with.

The spatial Fourier transform of a function $f(x)$ is given by
\begin{equation*}
\hat{f}(\xi)=\int_{\R^N}e^{-ix\cdot \xi}f(x)dx.
\end{equation*}

We shall also define $D^s$ and $J^s$ to be, respectively, the Fourier
multiplier with symbol $|\xi|^s$ and $\langle \xi \rangle^s = (1+|\xi|)^s$.
In this case, the norm in the Sobolev spaces $H^s(\R^N)$ and $\dot{H}^s(\R^N)$
are given, respectively, by
\begin{equation*}
\|f\|_{H^s(\R^N)}\equiv \|J^sf\|_{L^2_x(\R^N)}=\|\langle \xi
\rangle^s\hat{f}\|_{L^2_{\xi}(\R^N)} 
\end{equation*}
and
\begin{equation*}
\|f\|_{\dot{H}^s(\R^N)}\equiv
\|D^sf\|_{L^2_x(\R^N)}=\||\xi|^s\hat{f}\|_{L^2_{\xi}(\R^N)}.
\end{equation*}

Now, we prove the sharp Gagliardo-Nirenberg inequality \eqref{GN}.

\begin{proof}[Proof of Theorem \ref{best}]
We follow the ideas introduced by Weinstein \cite{W83}. First, define the Weinstein functional 
$$
J(u)=\dfrac{\|\nabla u\|_{L^2(\R^N)}^{N\sigma +b}\|u\|_{L^2(\R^N)}^{2\sigma+ 2-(N\sigma +b)}}{I(u)},
$$
where
$$
I(u)=\int_{\R^N} |x|^{-b}|u(x)|^{2\sigma+2}\;dx.
$$
It was proved in Genoud \cite{G2012} (see Lemma 2.1) that $I\in C(H^1(\R^N);\R)$ and is weakly sequentially continuous and $J\in C(H^1(\R^N) \backslash \{0\};\R)$. 

Now, since $J(u)\geq 0$, there exists a minimizing sequence $u_n \in H^1(\R^N)$ such that
$$
\lim_{n\rightarrow \infty}J(u)=m.
$$

By Schwarz symmetrization, we can assume that $u_n$ is radial and radially non-increasing for all $n$. Next, we rescale the sequence $\{u_n\}_{n\in \mathbb{N}}$ by setting $v_n(x)=\lambda_nu_n(\mu_n x)$ where
$$
\lambda_n=\dfrac{\|u_n\|^{N/2-1}_{L^2(\R^N)}}{\|\nabla u_n\|^{N/2}_{L^2(\R^N)}} \peq \textrm{ and } \peq \mu_n=\dfrac{\|u_n\|_{L^2(\R^N)}}{\|\nabla u_n\|_{L^2(\R^N)}} 
$$
so that $\|v_n\|_{L^2(\R^N)}=\|\nabla v_n\|_{L^2(\R^N)}=1$. Moreover, since $J$ is invariant under this scaling, $\{v_n\}_{n\in \mathbb{N}}$ is also a minimizing sequence which is bounded in $H^1(\R^N)$. Therefore, there exists $v^{\ast} \in H^1(\R^N)$ such that, up to a subsequence, $v_n \rightharpoonup v^{\ast}$ weakly in $H^1(\R^N)$. Furthermore, $v^{\ast}$ is non-negative, spherically symmetric, radially non-increasing, with
$$
\|v^{\ast}\|_{L^2(\R^N)}\leq 1 \peq \textrm{ and } \peq \|\nabla v^{\ast}\|_{L^2(\R^N)}\leq 1.
$$

In this case,
$$
m\leq J(v^{\ast}) \leq \dfrac{1}{I(v^{\ast})}=\lim_{n\rightarrow \infty} J(v_n)=m.
$$
Thus, $J(v^{\ast})=\dfrac{1}{I(v^{\ast})}=m$ and $\|v^{\ast}\|_{L^2(\R^N)}=\|\nabla v^{\ast}\|_{L^2(\R^N)}=1$. In particular $v^{\ast} \neq 0$ and $v_n \rightarrow v^{\ast}$ strongly in $H^1(\R^N)$. 

Therefore, $v^{\ast}$ is a minimizer for the Weinstein operator $J$. Moreover, $v^{\ast}$ is a solution of the Euler-Lagrange equation
$$
\dfrac{d}{d\varepsilon}J(v^{\ast} + \varepsilon \eta)|_{\varepsilon=0}=0,\peq \textrm{ for all } \eta \in C^{\infty}_0(\R^N)
$$
and so we obtain that $v^{\ast}$ satifies the equation
$$
\dfrac{N\sigma+2}{2}\Delta v^{\ast} - \dfrac{2\sigma+2-(N\sigma+b)}{2}v^{\ast}+m(\sigma+1)|v^{\ast}|^{2\sigma}v^{\ast}=0.
$$
Next, we rescale $v^{\ast}$ to a solution of equation \eqref{ground}. First, we take $\psi^{\ast}=[m(\sigma+1)]^{-1/2\sigma}v^{\ast}$. It is easy to see that $\psi^{\ast}$ is a solution of 
$$
\dfrac{N\sigma+2}{2}\Delta \psi^{\ast}- \dfrac{2\sigma+2-(N\sigma+b)}{2}\psi^{\ast}+|\psi^{\ast}|^{2\sigma}\psi^{\ast}=0.
$$
Furthermore, since $\|v^{\ast}\|_{L^2(\R^N)}=\|\nabla v^{\ast}\|_{L^2(\R^N)}=1$ we have
$$
m=\dfrac{\|\psi^{\ast}\|^{2\sigma}_{L^2(\R^N)}}{\sigma+1}=\dfrac{\|\nabla \psi^{\ast}\|^{2\sigma}_{L^2(\R^N)}}{\sigma+1} \peq \textrm{ and } \peq  \|\psi^{\ast}\|_{L^2(\R^N)}=\|\nabla \psi^{\ast}\|_{L^2(\R^N)}.
$$
Now set $Q(x)=\lambda\psi^{\ast}(\mu x)$, where 
$$
\lambda=\left[\left(\dfrac{N\sigma+b}{2\sigma+2-(N\sigma+b)}\right)^{\frac{2-b}{2}} \dfrac{2}{N\sigma+b}\right]^{1/2\sigma}\peq \textrm{ and } \peq \mu=\left(\dfrac{N\sigma+b}{2\sigma+2-(N\sigma+b)}\right)^{1/2} 
$$
so $Q$ is a solution of \eqref{ground} and 
\begin{equation}\label{groundR3}
\|Q\|_{L^2(\R^N)}=\dfrac{\lambda}{\mu^{N/2}}\|\psi^{\ast}\|_{L^2(\R^N)}.
\end{equation}

By the definition of $m$ and relation \eqref{groundR3} we have
$$
K_{\rm opt}= \dfrac{1}{m}=\dfrac{\sigma+1}{\|\psi^{\ast}\|^{2\sigma}_{L^2(\R^N)}} = \left(\frac{N\sigma+b}{2\sigma+2-(N\sigma+b)}\right)^{\frac{2-(N\sigma+b)}{2}} \frac{2\sigma+2}{(N\sigma+b)\|Q\|_{L^2}^{2\sigma}},
$$
which implies \eqref{opt1}.

To finish the proof we need to show the relations \eqref{groundR1} and \eqref{groundR2}. Indeed, the definition of $Q$ yields
\begin{equation}\label{groundR4}
\|\nabla Q\|_{L^2(\R^N)}=\dfrac{\lambda}{\mu^{\frac{N-2}{2}}}\|\nabla \psi^{\ast}\|_{L^2(\R^N)}.
\end{equation}

Moreover, since $\|\psi^{\ast}\|_{L^2(\R^N)}=\|\nabla \psi^{\ast}\|_{L^2(\R^N)}$ we obtain
$$
\|\nabla Q\|_{L^2(\R^N)}=\mu \| Q\|_{L^2(\R^N)},
$$
which implies \eqref{groundR1}.

On the other hand, by multiplying \eqref{ground} by $Q$ and integrating by parts we have
$$
\int_{\R^N} |x|^{-b}|Q(x)|^{2\sigma+2}\;dx=\|\nabla Q\|^2_{L^2(\R^N)}+ \| Q\|^2_{L^2(\R^N)}
$$
and using \eqref{groundR1} we conclude \eqref{groundR2}.

\end{proof}

%%%%%%%%%%%%%%%%%%%%%%%%%%%%%%%%%%%%%%%%%%%%%%%%%%%%%%%%%%%%%%%%%%%%%%%%%%%%%%%%%%%%%%%%%%%%%%%%%%%%%%%%%%%%%%%%%%%%%%%%%%%%
%%%%%%%%%%%%%%%%%%%%%%%%%%%%%%%%%%%%%%%%%%%%%%%%%%%%%%%%%%%%%%%%%%%%%%%%%%%%%%%%%%%%%%%%%%%%%%%%%%%%%%%%%%%%%%%%%%%%%%%%%%%%
%%%%%%%%%%%%%%%%%%%%%%%%%%%%%%%%%%%%%%%%%%%%%%%%%%%%%%%%%%%%%%%%%%%%%%%%%%%%%%%%%%%%%%%%%%%%%%%%%%%%%%%%%%%%%%%%%%%%%%%%%%%%
%%%%%%%%%%%%%%%%%%%%%%%%%%%%%%%%%%%%%%%%%%%%%%%%%%%%%%%%%%%%%%%%%%%%%%%%%%%%%%%%%%%%%%%%%%%%%%%%%%%%%%%%%%%%%%%%%%%%%%%%%%%%
%%%%%%%%%%%%%%%%%%%%%%%%%%%%%%%%%%%%%%%%%%%%%%%%%%%%%%%%%%%%%%%%%%%%%%%%%%%%%%%%%%%%%%%%%%%%%%%%%%%%%%%%%%%%%%%%%%%%%%%%%%%%
%%%%%%%%%%%%%%%%%%%%%%%%%%%%%%%%%%%%%%%%%%%%%%%%%%%%%%%%%%%%%%%%%%%%%%%%%%%%%%%%%%%%%%%%%%%%%%%%%%%%%%%%%%%%%%%%%%%%%%%%%%%%
%%%%%%%%%%%%%%%%%%%%%%%%%%%%%%%%%%%%%%%%%%%%%%%%%%%%%%%%%%%%%%%%%%%%%%%%%%%%%%%%%%%%%%%%%%%%%%%%%%%%%%%%%%%%%%%%%%%%%%%%%%%%
%%%%%%%%%%%%%%%%%%%%%%%%%%%%%%%%%%%%%%%%%%%%%%%%%%%%%%%%%%%%%%%%%%%%%%%%%%%%%%%%%%%%%%%%%%%%%%%%%%%%%%%%%%%%%%%%%%%%%%%%%%%%
%%%%%%%%%%%%%%%%%%%%%%%%%%%%%%%%%%%%%%%%%%%%%%%%%%%%%%%%%%%%%%%%%%%%%%%%%%%%%%%%%%%%%%%%%%%%%%%%%%%%%%%%%%%%%%%%%%%%%%%%%%%%

\section{Global Well-posedness}\label{global}

In this section we prove our main global well-posedness result.
 
\begin{proof}[Proof of Theorem \ref{global3}]
By the local theory, we just need to control the ${H}^1(\R^N)$ norm of $u(t)$ for all $t\in \R$. Using the quantities $M[u(t)]$ and $E[u(t)]$ and the sharp Gagliardo-Nirenberg inequality \eqref{GN} we have
\begin{equation}\label{ap10}
\begin{split}
2E[u_0] &= \|\nabla u(t)\|_{L^2(\R^N)}^2-\frac{1}{\sigma+1} \int_{\R^N} |x|^{-b}|u(x,t)|^{2\sigma+2}\;dx\\
&\geq \|\nabla u(t)\|_{L^2(\R^N)}^2-\frac{K_{\rm opt}}{\sigma+1}\,
\|u_0\|_{L^2(\R^N)}^{2\sigma+2-(N\sigma+b)}\|\nabla u(t)\|_{L^2(\R^N)}^{N\sigma+b}.
\end{split}
\end{equation}
Let $X(t)=\|\nabla u(t)\|_{L^2(\R^N)}^2$, $A=2E[u_0]$,
and $B=\frac{K_{\rm opt}}{\sigma+1}\|u_0\|_{L^2(\R^N)}^{2\sigma+2-(N\sigma+b)}$, then we can write \eqref{ap10} as
\begin{equation}\label{ap12}
X(t)-B\,X(t)^{\frac{N\sigma+b}{2}}\le A, \text{\hskip2pt for}\;\;t\in (0,T),
\end{equation}
where $T$ is the maximum time of existence given by the local theory.

Define the function $f(x)=x-B\,x^{\frac{N\sigma+b}{2}}$, for $x\ge 0$. Since $\sigma>(2-b)/N$ we have $deg(f)>1$. Moreover, a simple computation shows that $f$ has a local maximum at 
$$
x_0=\Big(\dfrac{2}{B(N\sigma+b)}\Big)^{2/N\sigma+b-2}
$$
with maximum value 
$$
f(x_0)=\dfrac{N\sigma+b-2}{N\sigma+b}\Big(\dfrac{2}{B(N\sigma+b)}\Big)^{2/N\sigma+b-2}.
$$

Using the relation \eqref{Eground}, the condition \eqref{GR1} implies that $2E[u_0] < f(x_0)$ which combining with \eqref{ap10} yields
\begin{equation}\label{ap13}
f(\|\nabla u(t)\|_{L^2(\R^N)}^2)\leq 2E(u_0) < f(x_0).
\end{equation}

Next, note that condition \eqref{GR2} is equivalent to $\|\nabla u(t)\|_{L^2(\R^N)}^2<x_0$. If initially it holds, then the continuity of $\|\nabla u(t)\|_{L^2(\R^N)}^2$ and \eqref{ap13} imply that 
$$
\|\nabla u(t)\|_{L^2(\R^N)}^2 < x_0
$$ 
for any t as long as the solution exists, which gives (\ref{GR3}). By mass conservation, we thus proved that the  ${H}^1(\R^N)$ norm of the solution $u(t)$ is bounded, which completes the proof of Theorem \ref{global3}.
\end{proof}

%%%%%%%%%%%%%%%%%%%%%%%%%%%%%%%%%%%%%%%%%%%%%%%%%%%%%%%%%%%%%%%%%%%%%%%%%%%%%%%%%%%%%%%%%%%%%%%%%%%%%%%%%%%%%%%%%%%%%%%%%%%%
%%%%%%%%%%%%%%%%%%%%%%%%%%%%%%%%%%%%%%%%%%%%%%%%%%%%%%%%%%%%%%%%%%%%%%%%%%%%%%%%%%%%%%%%%%%%%%%%%%%%%%%%%%%%%%%%%%%%%%%%%%%%
%%%%%%%%%%%%%%%%%%%%%%%%%%%%%%%%%%%%%%%%%%%%%%%%%%%%%%%%%%%%%%%%%%%%%%%%%%%%%%%%%%%%%%%%%%%%%%%%%%%%%%%%%%%%%%%%%%%%%%%%%%%%
%%%%%%%%%%%%%%%%%%%%%%%%%%%%%%%%%%%%%%%%%%%%%%%%%%%%%%%%%%%%%%%%%%%%%%%%%%%%%%%%%%%%%%%%%%%%%%%%%%%%%%%%%%%%%%%%%%%%%%%%%%%%
%%%%%%%%%%%%%%%%%%%%%%%%%%%%%%%%%%%%%%%%%%%%%%%%%%%%%%%%%%%%%%%%%%%%%%%%%%%%%%%%%%%%%%%%%%%%%%%%%%%%%%%%%%%%%%%%%%%%%%%%%%%%
%%%%%%%%%%%%%%%%%%%%%%%%%%%%%%%%%%%%%%%%%%%%%%%%%%%%%%%%%%%%%%%%%%%%%%%%%%%%%%%%%%%%%%%%%%%%%%%%%%%%%%%%%%%%%%%%%%%%%%%%%%%%
%%%%%%%%%%%%%%%%%%%%%%%%%%%%%%%%%%%%%%%%%%%%%%%%%%%%%%%%%%%%%%%%%%%%%%%%%%%%%%%%%%%%%%%%%%%%%%%%%%%%%%%%%%%%%%%%%%%%%%%%%%%%
%%%%%%%%%%%%%%%%%%%%%%%%%%%%%%%%%%%%%%%%%%%%%%%%%%%%%%%%%%%%%%%%%%%%%%%%%%%%%%%%%%%%%%%%%%%%%%%%%%%%%%%%%%%%%%%%%%%%%%%%%%%%
%%%%%%%%%%%%%%%%%%%%%%%%%%%%%%%%%%%%%%%%%%%%%%%%%%%%%%%%%%%%%%%%%%%%%%%%%%%%%%%%%%%%%%%%%%%%%%%%%%%%%%%%%%%%%%%%%%%%%%%%%%%%

\section{Virial type identities}\label{Virial}

This section is devoted to establish some Virial type identities for the solutions of equation \eqref{INLS}. Throughout the rest of the paper we assume that $u\in H^1(\R^N)\cap \{u: |x|u \in L^2(\R^N)\}$. We have the following result.

\begin{proposition}\label{PropMerle}
Let $u(x,t)$ be a solution of equation \eqref{INLS} and $T$ its maximum existence time. We have for all $t\in [0,T)$
\begin{equation}\label{VE1}
\frac{d}{dt}\int_{\R^N}|x|^2|u(x,t)|^2dx=4\textrm{Im}\int_{\R^N}\bar{u}(x,t)(\nabla u(x,t)\cdot x) dx
\end{equation}
and 

\begin{equation}\label{VE2}
\frac{d^2}{dt^2}\int_{\R^N}|x|^2|u(x,t)|^2dx=4\left[2\int_{\R^N}|\nabla u(x,t)|^2dx +\left(\dfrac{N-b}{\sigma+1}-N\right)\int_{\R^N} |x|^{-b}|u(x,t)|^{2\sigma+2}\;dx\right].
\end{equation}
\end{proposition}

\begin{remark}
In terms of $E[u]$ and $\int_{\R^N}|\nabla u(x,t)|^2dx$, relation \eqref{VE2} can be rewritten as
\begin{equation}\label{VE3}
\frac{d^2}{dt^2}\int_{\R^N}|x|^2|u(x,t)|^2dx=8(N\sigma+b)E[u_0]-4(N\sigma+b-2)\|\nabla u(t)\|_{L^2(\R^N)}^2.
\end{equation}
\end{remark}

\begin{remark}
For $b<0$, Chen and Guo \cite{CG07} also obtained Virial type identities for the INLS equation \eqref{INLS}. The idea of the proof is based on Merle \cite[Proposition 2.1]{M96}. For the sake of completeness we also give the proof of Proposition \ref{PropMerle} in our case below.
\end{remark}

\begin{proof}[Proof]
We follow closely the proof of Proposition 2.1 in Merle \cite{M96}. Multiplying the equation \eqref{INLS} by $2\bar{u}$ and taking the imaginary part we have
$$
\textrm{Im}(2i\partial_tu\bar{u})=-\textrm{Im}(2\Delta u \bar{u}).
$$
Note that  $\partial_t |u|^2=\textrm{Re}(2\partial_t u\bar{u})=\textrm{Im}(2i\partial_t u\bar{u})$, thus
$$
\partial_t |u|^2=-\textrm{Im}(2\Delta u \bar{u})=-\nabla \cdot (\textrm{Im}(\nabla u \bar{u})),
$$
where $\nabla \cdot f = \textrm{div}(f)=\partial f/\partial{x_1}+\cdots+\partial f/\partial{x_N}$.

Therefore
\begin{eqnarray*}
\frac{d}{dt}\int_{\R^N}|x|^2|u(x,t)|^2dx&=&-2\int_{\R^N} |x|^2\nabla \cdot (\textrm{Im}(\nabla u \bar{u}))dx\\
&=&4\int_{\R^N}\sum_{j=1}^N(x_j \cdot \textrm{Im}(\nabla u \bar{u}))dx\\
&=&4\textrm{Im}\int_{\R^N}\bar{u}(\nabla u \cdot x)dx\\
\end{eqnarray*}
which proves \eqref{VE1}.

On the other hand, taking time derivative in the previous relation, integrating by parts and using that $z-\bar{z}=2i\textrm{Im}(z)$ we obtain
\begin{equation}\label{VE4}
\begin{split}
\frac{d^2}{dt^2}\int_{\R^N}|x|^2|u(x,t)|^2dx&=4\textrm{Im}\left\{\int_{\R^N}(\bar{u}_t(\nabla u \cdot x)+\bar{u}(\nabla u_t \cdot x))dx\right\}\\
&=4\textrm{Im}\left\{\sum_{j=1}^N\int_{\R^N}\left(\bar{u}_t(\frac{\partial u}{\partial{x_j}} \cdot x_j)+\bar{u}(\frac{\partial u_t}{\partial{x_j}} \cdot x_j)\right)dx\right\}\\
&=4\textrm{Im}\left\{\sum_{j=1}^N\int_{\R^N}\left(\bar{u}_t(\frac{\partial u}{\partial{x_j}} \cdot x_j)-u_t (\frac{\partial \bar{u}}{\partial{x_j}} \cdot x_j)-\bar{u} u_t\right)dx\right\}\\
&=4\textrm{Im}\left\{2\int_{\R^N}\bar{u}_t(\nabla u \cdot x)dx- N\int_{\R^N} \bar{u} u_t dx\right\}.\\
\end{split}
\end{equation}

We study each term in the right hand side of \eqref{VE4} separately. For the second term, using the equation \eqref{INLS} and integrating by parts we have
\begin{equation}\label{VE5}
\begin{split}
-N\textrm{Im}\left\{\int_{\R^N} \bar{u} u_t dx\right\}
&=-N\textrm{Im}\left\{-i\int_{\R^N}|\nabla{u}|^2dx+i\int_{\R^N} |x|^{-b}|u|^{2\sigma+2} dx\right\}\\
&=N\int_{\R^N}|\nabla{u}|^2dx-N\int_{\R^N} |x|^{-b}|u|^{2\sigma+2} dx.\\
\end{split}
\end{equation}

Next, again using the equation \eqref{INLS}, the first term in the right hand side of \eqref{VE4} can be expressed as
\begin{equation}\label{VE6}
\begin{split}
2\textrm{Im}\left\{\int_{\R^N}\bar{u}_t(\nabla u \cdot x)dx\right\}
&=-2\textrm{Im}\left\{\int_{\R^N}\left[(i\Delta u\nabla\bar{u})\cdot x+(i |x|^{-b}|u|^{2\sigma}u)\nabla\bar{u}\cdot x\right] dx\right\}\\
&=-2\textrm{Re}\left\{\int_{\R^N}\left[\Delta u\nabla\bar{u}\cdot x+( |x|^{-b}|u|^{2\sigma}u)\nabla\bar{u}\cdot x\right] dx\right\}.\\
\end{split}
\end{equation}

Moreover, integration by parts yields	
\begin{equation*}
\begin{split}
\int_{\R^N}\Delta u\nabla\bar{u}\cdot x dx 
&=\sum_{j, k=1}^N \int_{\R^N}\frac{\partial^2 u}{\partial{x_k}^2} \left(\frac{\partial \bar{u}}{\partial{x_j}} x_j\right)dx\\
&=-\sum_{j, k=1}^N \int_{\R^N}\frac{\partial u}{\partial{x_k}} \frac{\partial^2 \bar{u}}{\partial{x_k} \partial{x_j}} x_jdx - \int_{\R^N}|\nabla u|^2dx.\\
\end{split}
\end{equation*}

Therefore, taking the real part and integrating by parts we obtain
\begin{equation}\label{VE7}
\begin{split}
\textrm{Re}\left\{\int_{\R^N}\Delta u\nabla\bar{u}\cdot x dx \right\}
&=-\frac{1}{2}\sum_{j, k=1}^N \int_{\R^N}\frac{\partial}{\partial{x_j}}\left(\frac{\partial u}{\partial{x_k}} \frac{\partial \bar{u}}{\partial{x_k}} \right) x_jdx - \int_{\R^N}|\nabla u|^2dx.\\
&=\frac{N-2}{2}\int_{\R^N}|\nabla u|^2dx.\\
\end{split}
\end{equation}

On the other hand, another integration by parts yields
\begin{equation}\label{VE8}
\begin{split}
\textrm{Re}\left\{\int_{\R^N} (|x|^{-b}|u|^{2\sigma}u)\nabla\bar{u}\cdot x dx \right\}
&=\frac{1}{2}\sum_{j=1}^N \int_{\R^N} x_j |x|^{-b}|u|^{2\sigma}\left(u\frac{\partial \bar{u}}{\partial x_j}+\bar{u} \frac{\partial u}{\partial x_j}\right)dx.\\
&=\frac{1}{2(\sigma+1)}\sum_{j=1}^N \int_{\R^N} x_j |x|^{-b}\frac{\partial}{\partial x_j}(|u|^{2\sigma+2})dx.\\
&=-\frac{1}{2(\sigma+1)}\int_{\R^N} \left(N|x|^{-b}|u|^{2\sigma+2} +x\cdot \nabla(|x|^{-b})|u|^{2\sigma+2}\right)dx.\\
\end{split}
\end{equation}

Next, collecting \eqref{VE4}-\eqref{VE8} we have
\begin{equation*}
\begin{split}
\frac{d^2}{dt^2}\int_{\R^N}|x|^2|u(x,t)|^2dx&=8\int_{\R^N}|\nabla|^2dx +4\left(\dfrac{N}{\sigma+1}-N\right)\int_{\R^N} |x|^{-b}|u|^{2\sigma+2}\;dx\\
&+\dfrac{4}{\sigma+1}\int_{\R^N} x\cdot \nabla(|x|^{-b})|u|^{2\sigma+2}\;dx.
\end{split}
\end{equation*}

Finally, noting that $x\cdot \nabla(|x|^{-b})=-b|x|^{-b}$ we finish the proof of relation \eqref{VE2}.

\end{proof}

%%%%%%%%%%%%%%%%%%%%%%%%%%%%%%%%%%%%%%%%%%%%%%%%%%%%%%%%%%%%%%%%%%%%%%%%%%%%%%%%%%%%%%%%%%%%%%%%%%%%%%%%%%%%%%%%%%%%%%%%%%%%
%%%%%%%%%%%%%%%%%%%%%%%%%%%%%%%%%%%%%%%%%%%%%%%%%%%%%%%%%%%%%%%%%%%%%%%%%%%%%%%%%%%%%%%%%%%%%%%%%%%%%%%%%%%%%%%%%%%%%%%%%%%%
%%%%%%%%%%%%%%%%%%%%%%%%%%%%%%%%%%%%%%%%%%%%%%%%%%%%%%%%%%%%%%%%%%%%%%%%%%%%%%%%%%%%%%%%%%%%%%%%%%%%%%%%%%%%%%%%%%%%%%%%%%%%
%%%%%%%%%%%%%%%%%%%%%%%%%%%%%%%%%%%%%%%%%%%%%%%%%%%%%%%%%%%%%%%%%%%%%%%%%%%%%%%%%%%%%%%%%%%%%%%%%%%%%%%%%%%%%%%%%%%%%%%%%%%%
%%%%%%%%%%%%%%%%%%%%%%%%%%%%%%%%%%%%%%%%%%%%%%%%%%%%%%%%%%%%%%%%%%%%%%%%%%%%%%%%%%%%%%%%%%%%%%%%%%%%%%%%%%%%%%%%%%%%%%%%%%%%
%%%%%%%%%%%%%%%%%%%%%%%%%%%%%%%%%%%%%%%%%%%%%%%%%%%%%%%%%%%%%%%%%%%%%%%%%%%%%%%%%%%%%%%%%%%%%%%%%%%%%%%%%%%%%%%%%%%%%%%%%%%%
%%%%%%%%%%%%%%%%%%%%%%%%%%%%%%%%%%%%%%%%%%%%%%%%%%%%%%%%%%%%%%%%%%%%%%%%%%%%%%%%%%%%%%%%%%%%%%%%%%%%%%%%%%%%%%%%%%%%%%%%%%%%
%%%%%%%%%%%%%%%%%%%%%%%%%%%%%%%%%%%%%%%%%%%%%%%%%%%%%%%%%%%%%%%%%%%%%%%%%%%%%%%%%%%%%%%%%%%%%%%%%%%%%%%%%%%%%%%%%%%%%%%%%%%%
%%%%%%%%%%%%%%%%%%%%%%%%%%%%%%%%%%%%%%%%%%%%%%%%%%%%%%%%%%%%%%%%%%%%%%%%%%%%%%%%%%%%%%%%%%%%%%%%%%%%%%%%%%%%%%%%%%%%%%%%%%%%

\section{Blow up in $H^1(\R^N)$}\label{blowup2}

We start this section with some preliminary results. The first one is a calculus fact.

\begin{lemma}\label{Parabola}
Let $f(x)=\frac{1}{2}x^2-ax^{\alpha}$, where $a>0$ and $\alpha>2$. Define $p(x)$ the tangent parabola at the positive local maximum of $f$, namely $(x_{\textrm{max}}, f(x_{\textrm{max}}))$, that pass through the positive root of $f$, namely $(x_{\textrm{root}},0)$ with $x_{\textrm{root}}>0$, then
$$
f(x)\geq p(x), \peq \textrm{for all} \peq x \in [x_{\textrm{max}}, x_{\textrm{root}}].
$$
\end{lemma} 

\begin{proof}[Proof]
Let $F_\alpha(x)=\frac 12x^2-\frac 1{\alpha} x^{\alpha}$ defined on $x\in [0,\infty)$, $B>0$ the unique positive number such that $F_{\alpha}(B)=0$ and $A=F_{\alpha}(1)$. Note that $A$ is the positive local maximum of $F_{\alpha}$.
Define $y=P_\alpha(x)$ the parabola with vertex $(1,A)$ and root $B$.

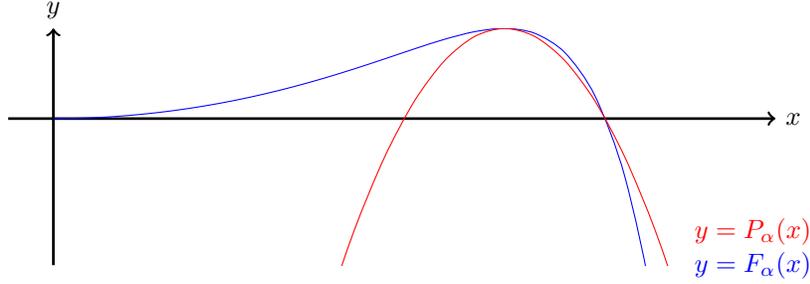
\begin{center}
  \begin{tikzpicture}[xscale=6,yscale=3,  domain=0:1.6, variable=\t, smooth]
\node[right, color=red] at (1.4,-0.5)  {$y=P_\alpha(x)$};  
\node[right, color=blue] at (1.4,-0.65)  {$y=F_\alpha(x)$};  
\draw[->,line width=1pt](-0.1,0)--(1.6,0) node[right] {$x$};  %eje x
  \draw[->,line width=1pt](0,-0.65)--(0,0.4) node[above] {$y$}; %eje y 
\clip (-0.2,-0.65) rectangle (2,0.4);
\draw[color=blue] plot(\t, {\t^2/2-\t^(10)/(10)}); %gráfico 
\draw [color=red]plot({\t}, {((1/2)-(1/10))*(1-(\t-1)^2/((10/2)^(1/(10-2))-1)^2)});
  \end{tikzpicture}
  \end{center}

It is clear the $A$ and $B$ can be given explicit in terms of $\alpha$. Indeed
$$
A=\frac 12-\frac 1\alpha \peq \textrm{ and } \peq B=\left({\frac \alpha 2}\right)^{\frac{1}{\alpha-2}}.
$$

Moreover, using the change of variables
$$
(x,y)\mapsto \left(\frac x{\sqrt[\alpha-2]{a\alpha}}, \frac y{\sqrt[\alpha-2]{a^2\alpha^2}}\right)
$$
we can reduce our problem to prove that $P_\alpha(x) \leq F_\alpha(x)$, for all $x\in [1,B]$.

Note that $y=F_\alpha(x)$ has an unique positive maximum at $(1,A)$, therefore the graphs of $y=F_\alpha$ and $y=P_\alpha$ are tangents at this point. Moreover, 
$$
P_\alpha(x)=A\left(1-\frac {(x-1)^2}{(B-1)^2)}\right).
$$

It is enough to prove that
\begin{equation}\label{PF1}
P_\alpha(x)<F_\alpha(x), \textrm{ for all } x \textrm{ close to } 1.
\end{equation}
and 
\begin{equation}\label{PF2}
P_\alpha(x)<F_\alpha(x), \textrm{ for all } x \textrm{ close and below } B.
\end{equation}

Indeed, define $G(x)=F_\alpha(x)-P_\alpha(x)$ and assume by contradiction that there exists $t_0\in (0,B)$ such that $G(t_0)<0$. If \eqref{PF1} and \eqref{PF2} hold, there exists $t_1$ close to $1$ and $t_2$ close and below $B$ such that $G(t_1)>0$ and $G(t_2)>0$. In this case, we can assume $1<t_1<t_0<t_2<B$. Therefore, there exist $r_1\in (t_1,t_0)$ and $r_2\in (t_0,t_2)$ such that $G(r_1)=G(r_2)=0=G(1)=G(B)$. By the Mean Value Theorem, $G'''(x)$ has at least one root in $(1,B)$. However $G'''(x)=-(\alpha-1)(\alpha-2) x^{\alpha-3}$ does not vanish in this interval, which is a contradiction.

To prove \eqref{PF1}, since the graphs are tangent at $(1,A)$, we just need to conclude 
$$
|P''_\alpha(1)|>|F''_\alpha(1)|.
$$

Using the expressions of $F_\alpha$ and $P_\alpha$ and taking $x=\alpha-2$ this is equivalent to 
$$
\left(1+\frac 1{\sqrt {x+2}}\right)^{x}\ge \frac {x+2}{2},
$$
which holds for all $x>0$.

On the other hand, since $F_\alpha(B)=P_\alpha(B)=0$, to obtain \eqref{PF2}, it is enough to prove
$$
F'_\alpha(B)<P'_\alpha(B).
$$ 

Again using the expressions of $F_\alpha$ and $P_\alpha$, the last inequality reduces to
$$
\frac B2>\frac 1{\alpha(B-1)}.
$$ 

Since $B=\left({\frac \alpha 2}\right)^{\frac{1}{\alpha-2}}$ and taking $x=\frac {\alpha-2}2$, this is equivalent to
$$
\displaystyle(1+x)^{\frac 1{2x}+1}\left((1+x)^{\frac 1{2x}}-1\right)>1,
$$
which holds for all $x>0$. 
\end{proof}

Next, we apply the previous result to prove an ``energy trapping'' inequality related to the (INLS) equation in the spirit of Kenig and Merle \cite{KM06} and Cazenave, Fang and Xie \cite{CFX}.
\begin{proposition}\label{Parabola2}
Let $u\in H^1(\R^N)$ and $Q$ be the unique non-negative, radially-symmetric, decreasing solution of the equation \eqref{ground}. Then
\begin{itemize}
\item[(a)] If $E[u]\leq 0$ then 
$$\|\nabla u\|_{L^2(\R^N)}^{s_{\sigma}}\|u\|_{L^2(\R^N)}^{1-s_{\sigma}} \geq c_{\sigma,b,N}\|\nabla Q\|_{L^2(\R^N)}^{s_{\sigma}}\|Q\|_{L^2(\R^N)}^{1-s_{\sigma}},
$$
where $c_{\sigma,b,N}=\left(\dfrac{N\sigma+b}{2}\right)^{1/(N\sigma+b-2)}.$

\item[(b)]If $E[u] > 0$ and $E[u]^{s_{\sigma}} M[u]^{1-s_{\sigma}} < E[Q]^{s_{\sigma}} M[Q]^{1-s_{\sigma}}$ 
then 
$$\|\nabla u\|_{L^2(\R^N)}^{s_{\sigma}}\|u\|_{L^2(\R^N)}^{1-s_{\sigma}} > c_{\sigma,b,N}\|\nabla Q\|_{L^2(\R^N)}^{s_{\sigma}}\|Q\|_{L^2(\R^N)}^{1-s_{\sigma}},
$$
where 
$$
c_{\sigma,b,N,Q,u}=\left(1+\left(1-\dfrac{E[u] M[u]^{\frac{s_{\sigma}}{1-s_{\sigma}}}}{E[Q]M[Q]^{\frac{s_{\sigma}}{1-s_{\sigma}}}}\right)^{1/2}\left(\left(\dfrac{N\sigma+b}{2}\right)^{1/(N\sigma+b-2)}-1\right)\right)^{s_\sigma}.
$$
\end{itemize}
\end{proposition}

\begin{remark}
Note that $c_{\sigma,b,N}, c_{\sigma,b,N,Q,u}>1$ since $\sigma>\dfrac{2-b}{N}$. 
\end{remark}

\begin{proof}[Proof]
Recalling the definition of $E[u]$ and multiplying both sides by $M[u]^{\frac{s_{\sigma}}{1-s_{\sigma}}}$ we obtain
\begin{equation*}
\begin{split}
E[u] M[u]^{\frac{s_{\sigma}}{1-s_{\sigma}}}&= \frac{1}{2}\left(\|\nabla u\|_{L^2(\R^N)}\| u\|_{L^2(\R^N)}^{\frac{s_{\sigma}}{1-s_{\sigma}}}\right)^2-\frac{1}{2\sigma+2} \| u\|_{L^2(\R^N)}^{\frac{2s_{\sigma}}{1-s_{\sigma}}}\int_{\R^N} |x|^{-b}|u(x|^{2\sigma+2}\;dx\\
&\geq \frac{1}{2}\left(\|\nabla u\|_{L^2(\R^N)}\| u\|_{L^2(\R^N)}^{\frac{s_{\sigma}}{1-s_{\sigma}}}\right)^2 -\frac{K_{\rm opt}}{2\sigma+2}\,
\left(\|\nabla u\|_{L^2(\R^N)}\| u\|_{L^2(\R^N)}^{\frac{s_{\sigma}}{1-s_{\sigma}}}\right)^{N\sigma+b}.
\end{split}
\end{equation*}

Therefore
\begin{equation}\label{EBU1}
E[u] M[u]^{\frac{s_{\sigma}}{1-s_{\sigma}}}\geq f(\|\nabla u\|_{L^2(\R^N)}\| u\|_{L^2(\R^N)}^{\frac{s_{\sigma}}{1-s_{\sigma}}}),
\end{equation}
where $f(x)=\frac{1}{2}x^2-\frac{K_{\rm opt}}{2\sigma+2}x^{N\sigma+b}$

Some straightforward computations revel that $f$ has a local maximum in 
$$
x_{\textrm{max}}=\|\nabla Q\|_{L^2(\R^N)}\| Q\|_{L^2(\R^N)}^{\frac{s_{\sigma}}{1-s_{\sigma}}}
$$
with maximum value 
$$
f(x_{\textrm{max}})=E[Q] M[Q]^{\frac{s_{\sigma}}{1-s_{\sigma}}}.
$$ 

Moreover $f$ has a positive root, denoted by $x_{\textrm{root}}$, where
$$
x_{\textrm{max}}<x_{\textrm{root}}=\left(\dfrac{N\sigma+b}{2}\right)^{1/(N\sigma+b-2)}\|\nabla Q\|_{L^2(\R^N)}\| Q\|_{L^2(\R^N)}^{\frac{s_{\sigma}}{1-s_{\sigma}}}.
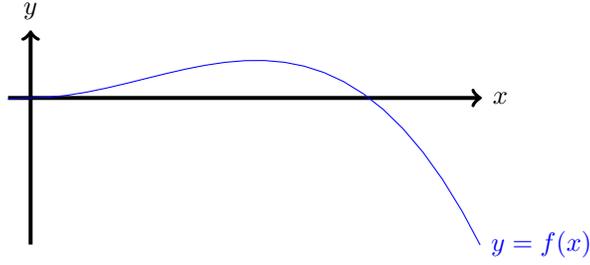
$$

\begin{center}
\begin{tikzpicture}[xscale=3,yscale=3,  domain=-0.1:2, variable=\t]
%\node[right, color=red] at (2,-0.5)  {$y=P_\alpha(x)$};  
\node[right, color=blue] at (2,-0.65)  {$y=f(x)$};  
\draw[->,line width=1.5pt](-0.1,0)--(2,0) node[right] {$x$};  %eje x
\draw[->,line width=1.5pt](0,-0.65)--(0,0.3) node[above] {$y$}; %eje y 
\clip (-0.2,-0.65) rectangle (2,0.3);
\draw[color=blue] plot(\t, {\t^2/2-\t^3/3}); %gráfico 
%\draw [color=red]plot({\t}, {1/6-2*(\t-1)^2/3});
\end{tikzpicture}
\end{center}

Now if $E[u]\leq 0$ than $f(\|\nabla u\|_{L^2(\R^N)}\| u\|_{L^2(\R^N)}^{\frac{s_{\sigma}}{1-s_{\sigma}}})\leq 0$, so its is clear that (a) holds. Next we turn our attention to the proof of part (b). To simplify our notation, for any function $\phi \in H^1(\R^N)$,  let us define the following quantities
$$
|E\phi|= E[\phi] M[\phi]^{\frac{s_{\sigma}}{1-s_{\sigma}}} \peq \textrm{ and } \peq
|\phi|=\|\nabla u\|_{L^2(\R^N)}\| u\|_{L^2(\R^N)}^{\frac{s_{\sigma}}{1-s_{\sigma}}}.
$$

If $|u|>x_{\textrm{root}}$ we are done. On the other hand, if 
$$
|Q|= x_{\textrm{max}}<|u|<x_{\textrm{root}}=\left(\dfrac{N\sigma+b}{2}\right)^{1/(N\sigma+b-2)}|Q|
$$ 
let $y=p(x)$ be the parabola with vertex $(|Q|,|EQ|)$ and root $x_r$. The equation of $y=p(x)$ is given explicit by 
$$
p(x)=a(x-|Q|)^2+|EQ|,
$$
where $a$ is given by the relation $p(x_r)=0$, that is 
\begin{equation}\label{a}
a=-\dfrac{|EQ|}{(x_r-|Q|)^2}.
\end{equation}

By Lemma \ref{Parabola} and inequality \eqref{EBU1} we have
$$
|Eu|\geq f(|u|)\geq p(|u|)= a(|u|-|Q|)^2+|EQ|,
$$
which together with the definition of $a$ (see relation \eqref{a}) yields part (b).
\end{proof}

Now, we have all tools to prove our blow-up result.

\begin{proof}[Proof of Theorem \ref{blowup}]
Suppose by contradiction that the solution $u(t)$ of equation \eqref{INLS} with initial data satisfying hypotheses \eqref{BR1}-\eqref{BR2} exists globally. Multiplying the Virial identity \eqref{VE3} by $M[u]^{\frac{s_{\sigma}}{1-s_{\sigma}}}$ and using Proposition \ref{Parabola2} we have for all $t>0$
\begin{equation*}
\begin{split}
\left(\frac{d^2}{dt^2}\int_{\R^N}|x|^2|u(x,t)|^2dx\right)M[u_0]^{\frac{s_{\sigma}}{1-s_{\sigma}}}&=8(N\sigma+b)E[u_0]M[u_0]^{\frac{s_{\sigma}}{1-s_{\sigma}}}\\
&-4(N\sigma+b-2)\left(\|\nabla u(t)\|_{L^2(\R^N)}\|u_0\|_{L^2(\R^N)}^{\frac{s_{\sigma}}{1-s_{\sigma}}}\right)^2\\
&<8(N\sigma+b)E[Q]M[Q]^{\frac{s_{\sigma}}{1-s_{\sigma}}}\\
&-4(N\sigma+b-2)A\left(\|\nabla Q\|_{L^2(\R^N)}\|Q\|_{L^2(\R^N)}^{\frac{s_{\sigma}}{1-s_{\sigma}}}\right)^2,
\end{split}
\end{equation*}
for some number $A=A(\sigma,b,N,Q,u_0)>1$, given by Proposition \ref{Parabola2}.

Recalling relation \eqref{Eground} it is easy to see that
$$
8(N\sigma+b)E[Q]M[Q]^{\frac{s_{\sigma}}{1-s_{\sigma}}}=4(N\sigma+b-2)\left(\|\nabla Q\|_{L^2(\R^N)}\|Q\|_{L^2(\R^N)}^{\frac{s_{\sigma}}{1-s_{\sigma}}}\right)^2
$$

Therefore
\begin{equation}\label{blowup3}
\begin{split}
\left(\frac{d^2}{dt^2}\int_{\R^N}|x|^2|u(x,t)|^2dx\right)M[u_0]^{\frac{s_{\sigma}}{1-s_{\sigma}}}&<-4(N\sigma+b-2)(A-1)\left(\|\nabla Q\|_{L^2(\R^N)}\|Q\|_{L^2(\R^N)}^{\frac{s_{\sigma}}{1-s_{\sigma}}}\right)^2\\
&=-B,
\end{split}
\end{equation}
for some number $B=B(\sigma,b,N,Q,u_0)>0$.

Finally, integrating \eqref{blowup3} twice and taking $t$ large we reach a contradiction.

\end{proof}

%%%%%%%%%%%%%%%%%%%%%%%%%%%%%%%%%%%%%%%%%%%%%%%%%%%%%%%%%%%%%%%%%%%%%%%%%%%%%%%%%%%%%%%%%%%%%%%%%%%%%%%%%%%%%%%%%%%%%%%%%%%%
%%%%%%%%%%%%%%%%%%%%%%%%%%%%%%%%%%%%%%%%%%%%%%%%%%%%%%%%%%%%%%%%%%%%%%%%%%%%%%%%%%%%%%%%%%%%%%%%%%%%%%%%%%%%%%%%%%%%%%%%%%%%
%%%%%%%%%%%%%%%%%%%%%%%%%%%%%%%%%%%%%%%%%%%%%%%%%%%%%%%%%%%%%%%%%%%%%%%%%%%%%%%%%%%%%%%%%%%%%%%%%%%%%%%%%%%%%%%%%%%%%%%%%%%%
%%%%%%%%%%%%%%%%%%%%%%%%%%%%%%%%%%%%%%%%%%%%%%%%%%%%%%%%%%%%%%%%%%%%%%%%%%%%%%%%%%%%%%%%%%%%%%%%%%%%%%%%%%%%%%%%%%%%%%%%%%%%
%%%%%%%%%%%%%%%%%%%%%%%%%%%%%%%%%%%%%%%%%%%%%%%%%%%%%%%%%%%%%%%%%%%%%%%%%%%%%%%%%%%%%%%%%%%%%%%%%%%%%%%%%%%%%%%%%%%%%%%%%%%%
%%%%%%%%%%%%%%%%%%%%%%%%%%%%%%%%%%%%%%%%%%%%%%%%%%%%%%%%%%%%%%%%%%%%%%%%%%%%%%%%%%%%%%%%%%%%%%%%%%%%%%%%%%%%%%%%%%%%%%%%%%%%
%%%%%%%%%%%%%%%%%%%%%%%%%%%%%%%%%%%%%%%%%%%%%%%%%%%%%%%%%%%%%%%%%%%%%%%%%%%%%%%%%%%%%%%%%%%%%%%%%%%%%%%%%%%%%%%%%%%%%%%%%%%%
%%%%%%%%%%%%%%%%%%%%%%%%%%%%%%%%%%%%%%%%%%%%%%%%%%%%%%%%%%%%%%%%%%%%%%%%%%%%%%%%%%%%%%%%%%%%%%%%%%%%%%%%%%%%%%%%%%%%%%%%%%%%
%%%%%%%%%%%%%%%%%%%%%%%%%%%%%%%%%%%%%%%%%%%%%%%%%%%%%%%%%%%%%%%%%%%%%%%%%%%%%%%%%%%%%%%%%%%%%%%%%%%%%%%%%%%%%%%%%%%%%%%%%%%%

\subsection*{Acknowledgments}
The author was partially supported by Conselho Nacional de Desenvolvimento Cient\'ifico e Tecnol\'ogico (CNPq/Brazil) and Funda\c{c}\~ao de Amparo \`a Pesquisa do Estado de Minas Gerais (FAPEMIG/Brazil). The author thanks F\'abio Brochero (UFMG) for valuable assistance with the proof of Lemma \ref{Parabola} and for his support in drafting the figures included in this paper. 

%The author thanks the referee for helpful comments and suggestions which improved the presentation of the paper.

\bibliographystyle{mrl}

\end{document}